\documentclass[a4paper,oneside,reqno]{amsart}

\usepackage{notations}
\usepackage{style}
\usepackage{packages}

\input{toc_mod.tex}

\title{Combinatorics of line arrangements and dynamics of polynomial vector fields}
\author{Beno\^it Guerville-Ball\'e}
\address{
	Insitut Joseph Fourier 
	UMR 5582 CNRS-UJF  
	100 rue des Math\'ematiques - BP 74  
	38 402 Saint-Martin-d'Heres Cedex 
}
\email{benoit.guerville-balle@ujf-grenoble.fr}

\author{Juan Viu-Sos}
\address{
	Laboratoire de Math\'ematiques et de leurs Applications
	UMR CNRS 5142
	B\^atiment IPRA - Universit\'e de Pau et des Pays de l'Adour
	Avenue de l'Universit\'e - BP 1155
	64013 PAU CEDEX
}
\email{juan.viusos@univ-pau.fr}
\thanks{Both author's are partially supported by the ANR Project Interlow ANR-09-JCJC-0097-01 and the JSPS-MAE PHC-Sakura 2014 Project}				
\subjclass[2010]{52C30, 34C07, 34C08}		
\date{}

\begin{document}

\begin{abstract}
	Let $\A$ be a real line arrangement and $\D(\A)$ the module of $\A$--derivations. First, we give a dynamical interpretation of $\D(\A)$ as the set of polynomial vector fields which posses $\A$ as invariant set. We characterize polynomial vector fields having an infinite number of invariant lines. Then we prove that the minimal degree of polynomial vector fields fixing only a finite set of lines in $\D(\A)$ is not determined by the combinatorics of $\A$.
\end{abstract}

\maketitle

\tableofcontents
\bigskip
\bigskip

\section{Introduction}\label{sec:intro}

A \emph{real line arrangement} $\A$ is a finite set $\{L_1,\ldots,L_n\}$ of lines in $\RR^2$. Its combinatorial data is encoded in the \emph{intersection poset} $L(\A)=\{\emptyset\neq L_i\cap L_j \mid L_i, L_j\in\A\}\cup\A$ partially ordered by reverse inclusion of subsets. The \emph{module of $\A$--derivations}, denoted by $\D(\A)$, is a classical algebraic geometric object associated with an arrangement introduced by Saito in~\cite{Saito80} in a more general context. It is usually studied from an algebraic point of view, but it also has a dynamical interpretation: $\D(\A)$ can be identified with the set of polynomial vector fields in $\RR^2$ possessing $\A$ as invariant set (i.e. the \emph{logarithmic vector fields of the arrangement}). We use this point of view in the following.

The influence of combinatorics of line arrangements over the properties of its realizations into different ambient spaces (as $\RR^2$, $\CC^2$, $\FF_p^2$ and their projectives) was largely studied, e.g. \cite{Arnold69}, \cite{OrlikSolomon80}, \cite{Rybnikov11}. Our work consists in the study of the relation between $\A$, the poset $L(\A)$, and the module $\D(\A)$, in order to understand the influence of the combinatorial structure on the minimal degree of vector fields in $\D(\A)$ and their corresponding dynamics in the real plane.

As a first step, a characterization of the polynomial vector fields admitting only a finite number of invariant lines is required. Then we investigate the minimal degree $d_f(\A)$ of logarithmic vector fields of this kind. We obtain lower bounds of this minimal degree depending only of the combinatorics of $\A$. Finally, we prove that even if $d_f(\A)$ admits a combinatorial lower bound, it is not determined by the intersection poset $L(\A)$.

This approach contrasts with the classical ones given in dynamical systems in the study of invariant lines in systems of low fixed degree by Llibre et al. and Xiang (\cite{Llibre06},\cite{Xiang98}). We are also far away from a more algebraic point of view, as in the works of Abe, Vall\`es and Faenzi (\cite{DanieleJean12},\cite{DanieleJean14},\cite{AbeDanieleJean14}), in terms of logarithmic bundles on the complex projective plane.\\
Following this rapprochement, we would be able to give an interpretation in the real plane of the Terao's conjecture, which asks about the combinatoriality of $\D(\A)$ for free central arrangements.

This paper is organized as follows: in Section~\ref{sec:preliminaries}, we recall the construction of the module of $\A$--derivations, we give its dynamical interpretation and we obtain a structure theorem on the set of logarithmic vector fields with bounded degree. In Section~\ref{sec:finiteness}, we give a characterization of vectors fields in $\D(\A)$ fixing only an infinite number of lines in $\RR^2$. We also study how the maximal multiplicity of singularities and maximal number of parallel lines in $\A$ give a lower bound for $d_f(\A)$. We prove in Section~\ref{sec:combinatorially} that $d_f(\A)$ does not depend on the number of lines and singularities counted by multiplicity (i.e. \emph{weak combinatorics}) or on the intersection poset (i.e. \emph{strong combinatorics}) of the line arrangement~$\A$, using two explicit counter-examples (Ziegler and Pappus arrangements). Finally, we give in Section~\ref{sec:conclu} some perspectives to work in the direction of the Terao's conjecture.

\section{The module of $\A$--derivations}\label{sec:preliminaries}

\subsection{The module of $\A$--derivations and planar vector fields}

Let $S=\Sym(\RR^2)^*$ be the symmetric algebra of the dual space of $\RR^2$. Taking $\{x,y\}$ the dual basis of the canonical one in $\RR^2$, we may identify $S$ with $\RR[x,y]$. For a line $L\in\A$, consider $\alpha_L:\RR^2\to\RR$ an associated affine form such that $L=\ker\alpha_L$. A defining polynomial of $\A$ is given by $\Q_\A=\prod_{L\in\A} \alpha_L$. Let $\Der_{\RR}(\RR[x,y])$ be the algebra of $\RR$-derivations of $\RR [x,y]$,  the \emph{module of $\A$--derivations} (also called \emph{module of logarithmic derivations of $\A$}) is the $\RR[x,y]$-module defined by:
\[
	\D(\A)=\{\chi\in\Der_{\RR}(\RR[x,y]) \mid \chi\Q_\A\in \I_{\Q_\A}\}
\]
where $\I_{\Q_\A}$ is the ideal generated by $\Q_\A$. From the previous definition of $\D(\A)$, it is easy to deduce that $\D(\A)=\bigcap_{L\in\A}\{\chi\in\Der_{\RR}(\RR[x,y]) \mid \chi\alpha_L\in\I_{\alpha_L}\}$, where $\I_{\alpha_L}$ is the ideal generated by $\alpha_L$.

Consider a real planar polynomial differential system defined for $(x,y)\in\RR^2$ by
\begin{equation*}\label{eq:dsystem}
    \frac{dx}{dt}=P(x,y)\qquad\frac{dy}{dt}=Q(x,y)
\end{equation*}
where $P,Q\in\RR[x,y]$. This globally defined autonomous system is associated to a polynomial vector field in the plane given by
\begin{equation}\label{eqn:vector_field}
    \chi=P\partial_x + Q\partial_y
\end{equation}
Following the language of dynamical systems, a polynomial vector field $\chi$ is considered of \emph{degree $d$} if $\max\{\deg P,\deg Q\}=d$.

Since $\Der_{\RR}(\RR[x,y])$ is in correspondence with polynomials vector fields on the plane, we obtain a dynamical interpretation of $\D (\A )$:
\begin{lem}
	Let $\A$ be an arrangement and $\chi\in\Der_{\RR}(\RR[x,y])$. Then $\chi\in\D(\A)$ if and only if $\A$ is invariant by $\chi$.
\end{lem}

\begin{rmk}
	A set $X\subset \RR^2$ is invariant by $\chi$ if $\phi_t(X)\subset X$, for every $t\in\RR$, where $\phi_t:\RR^2\to\RR^2$ is the flow associated to $\chi$ at instant $t$. 
\end{rmk}

In the case of real line arrangements, the required condition for a derivation to belong to $\D(\A)$ is equivalent to the definition of algebraic invariant sets in complex dynamical systems: a complex algebraic curve $\C=\{f=0\}$ is invariant by a polynomial vector field $\chi$ if there exists $K\in\CC[x,y]$ such that $\chi f=Kf$ (see~\cite{Dumortier06}).

%

The notion of degree of polynomial vector fields gives a natural filtration of the module of derivations.
\begin{equation*}
\Der_{\RR}(\RR[x,y])=\bigcup_{d\in\ZZ} \Der_{\RR}(\RR_d[x,y]),
\end{equation*}
where $\Der_{\RR}(\RR_d[x,y])=\{P\partial_x+Q\partial_y \mid \deg P,\deg Q\leq d\}$. Note that $\Der_{\RR}(\RR_d[x,y])=\emptyset$, for $d<0$. Restricting to the module of derivations, we obtain an ascending filtration of $\D(\A)$ by the vectorial spaces $\F_d\D(\A)=\D(\A)\cap\Der_{\RR}(\RR_d[x,y])$. We denote by $\D_d(\A)=\F_d\D(\A)\setminus \F_{d-1}\D(\A)$ the set of \emph{polynomial vector fields of degree $d$ fixing $\A$}.

\subsection{Geometry of logarithmic vector fields}\label{subsec:log_vect_field}

We begin with a necessary and sufficient condition on a line to be invariant by a polynomial vector field.

\begin{propo}\label{Proposition_CNS}
	Let $L$ be a line of $\RR^2$ defined by the equation $f=\alpha x +\beta y + \gamma = 0$, and let $\chi=P(x,y)\partial_x + Q(x,y) \partial_y$ be a polynomial vector field on $\RR^2$. The line $L$ is invariant for $\chi$ if and only if we are in one of the following cases:
	\begin{enumerate}
		 \item $\beta=0$ and $P(-\gamma/\alpha,y)=0$,
		 \item $\beta\neq 0$ and $\alpha P(\beta y, -\alpha y+\gamma/\beta) + \beta Q(\beta y, -\alpha y+\gamma/\beta) = 0$.
		\end{enumerate}
\end{propo}

\begin{proof}
	It is easy to check that the vertical line $L_{x}=\{x=0\}$ is invariant by $\chi$ if and only if $P(0,y)=0$. In order to obtain the result, we make an affine transformation of the plane $\varphi:\RR^2\rightarrow\RR^2$ such that $L$ is sent on the line $L_x$. The vector field $\chi$ can be seen as a section of the tangent bundle $T(\RR^2)$ and we denote by $\chi_\varphi$ the pushforward of $\chi$ by $\varphi$.
	\begin{equation*}
		\begin{tikzcd}
			T(\RR^2) 	\arrow{r}{\varphi_*}   \arrow{d}									 												& T(\RR^2)	\arrow{d}  \\[1.5em] 
			\RR^2     		\arrow{r}{\varphi}			\arrow[bend left=20]{u}{\chi} 		 					& \RR^2 		\arrow[bend right=20]{u}[swap]{\chi_\varphi}
		\end{tikzcd}
	\end{equation*}
	Hence, $L=\{f=0\}$ is invariant by $\chi$ if and only if $L_x=\{x=0\}$ is invariant by $\chi_\varphi$. 
	
	Assume $\beta=0$, thus $L$ is vertical and only a translation $\varphi(x,y)=(x-\gamma/\alpha,y)$ is needed:
	\begin{equation*}
		\chi_\varphi=P(x-\gamma/\alpha,y)\partial_x + Q(x-\gamma/\alpha,y)\partial_y
		\end{equation*}
	If $\beta\neq 0$, we consider $\varphi(x,y)=(\alpha x+\beta y,\beta x-\alpha y+\gamma/\beta)$ and we obtain: 
	\begin{multline*}
		\chi_\varphi=\left[ \alpha P(\alpha x+\beta y,\beta x-\alpha y+\gamma/\beta) + \beta Q(\alpha x+\beta y,\beta x-\alpha y+\gamma/\beta) \right]\partial_x + \\
		\left[ \beta P(\alpha x+\beta y,\beta x-\alpha y+\gamma/\beta) - \alpha Q(\alpha x+\beta y,\beta x-\alpha y+\gamma/\beta) \right]\partial_y.
	\end{multline*}
	Clearly, $L_x$ is invariant by $\chi_\varphi$ if and only if the coordinate of $\partial_x$ in $\chi_\varphi$ is zero. This implies the result.
\end{proof}

\begin{rmk}\label{Rmk_EquationCNS}
	Considering the vector field $\chi=P(x,y)\partial_x + Q(x,y)\partial_y$ of degree $d$ defined by generic polynomials 
    \begin{equation}\label{eq:generic_pols}
		P(x,y)=\sum\limits_{i+j\leq d} a_{i,j}x^iy^j \quad \text{ and } \quad Q(x,y)=\sum\limits_{i+j\leq d} b_{i,j}x^iy^j,
	\end{equation}
	with real coefficients, we can express the LHS of the equation of Proposition~\ref{Proposition_CNS} case (2), as a univariate polynomial in $\RR[y]$ in terms of $P$ and $Q$:
    \begin{equation*}
        R(y)=P(\beta y, -\alpha y+\gamma/\beta) + \beta Q(\beta y, -\alpha y+\gamma/\beta)
    \end{equation*}
	Thus, in the case $\beta\neq0$ the reader can easily verify that the equation $R(y)=0$ is equivalent to the system composed by:
    \begin{equation*}\label{eq:coeffs_Y}\tag{$E_m$}
        \quad0=\text{Coeff}_{y^m}R(y)=\sum_{k=0}^{d-m}\sum_{l=0}^m (\alpha a_{m-l,k+l} + \beta b_{m-l,k+l})\cdot \binom{k+l}{k}\cdot(-\alpha)^l\beta^{m-k-l}\gamma^k
    \end{equation*}
    for every $0\leq m\leq \deg{R}$.
\end{rmk}

Consider $\C(d)$ the $\RR$-linear space of coefficients of a pair of polynomials of degree less or equal than $d$, as in equation (\ref{eq:generic_pols}). We have $\C(d)=\RR^{(d+1)(d+2)/2}\oplus\RR^{(d+1)(d+2)/2}\simeq\RR^{(d+1)(d+2)}$. Fixing a line arrangement $\A$, we get by Proposition~\ref{Proposition_CNS} and Remark~\ref{Rmk_EquationCNS} that the equations defining $\F_d\D(\A)$ are linear in the coefficients of $P$ and $Q$, thus we can compute $\F_d\D(\A)$ as kernel of a linear map $\psi: \C(d)\to\RR^{n(d+1)}$, where $|\A|=n$.

\begin{thm}[\bf Structure of polynomial vector fields]\label{thm:struc_vf}
	Let $\A$ be a line arrangement. For each $d\in\NN$, $\F_d\D(\A)$ is a linear sub-space of the space of coefficients $\C(d)$.
\end{thm}
\section{Finiteness of derivations and combinatorial data}\label{sec:finiteness}

\subsection{Finiteness of fixed families of lines}

In order to efficiently characterize line arrangements as invariant sets of a polynomial vector field, the first step is to obtain conditions on the finiteness of the family of invariant lines under a vector field. This leads us to the notion of maximal line arrangement fixed by a polynomial vector field.

\begin{de}
  Let $\chi$ be a polynomial vector field in the plane. We said that a line arrangement $\A$ is \emph{maximal fixed by} $\chi$ if any line $L\subset\RR^2$ invariant by $\chi$ belongs to $\A$.
\end{de}

\begin{rmk}
    The notion of line arrangement is taken generally considering a finite collection of lines. Thus, there exist polynomial vector fields in the plane for which there are no a maximal line arrangements fixed by them: the null vector field is a trivial example, as well as a ``central'' vector field $\chi_c=x\partial_x + y\partial_y$ or a ``parallel'' vector field $\chi_p=(x+1)\partial_y$.\\
    In Theorem~\ref{thm:infty_vf} we prove that the derivations which have not maximal fixed arrangements are essentially of these types.
	\begin{figure}[ht]
	    \centering
	    \begin{tikzpicture}[scale=0.5]
	\begin{scope}[xshift=-6cm]
	\draw[->] (-4.5,0) -- (4.5,0);
	\draw[->] (0,-4.5) -- (0,4.5);
	\foreach \theta in {0,15,...,360} {
		\draw[color=blue] (\theta:4.25) -- (\theta:-4.25);
	}
	\foreach \rho in {0,1,...,4} {
		\foreach \theta in {0,15,...,360} {
			\draw[->,color=blue] (\theta:\rho) -- (\theta:\rho+0.1);
		}
	}
	\end{scope}
	\begin{scope}[xshift=6cm]
	\draw[->] (-4.5,0) -- (4.5,0);
	\draw[->] (0,-4.5) -- (0,4.5);
	\foreach \x in {-4,-3.25,...,4.25} {
		\draw[color=blue] (\x,4.25) -- (\x,-4.25);
	}
	\foreach \x in {-4,-3.25,...,-1.5} {
		\foreach \y in {-4,-3.5,...,4} {
			\draw[->, color=blue] (\x,\y) -- (\x,\y-0.1);
		}
	}
	\foreach \y in {-4,-3.75,...,4} {
			\node[color=blue] at (-1,\y) {$\cdot$};
	}
	\foreach \x in {-0.25,0.5,...,4.25} {
		\foreach \y in {-4,-3.5,...,4} {
			\draw[->, color=blue] (\x,\y) -- (\x,\y +0.1);
		}
	}
	\end{scope}
\end{tikzpicture}
	    \caption{Phase portraits of polynomial vector fields $\chi_c=x\partial_x + y\partial_y$ and $\chi_p=(x+1)\partial_y$, respectively.}
	\end{figure}
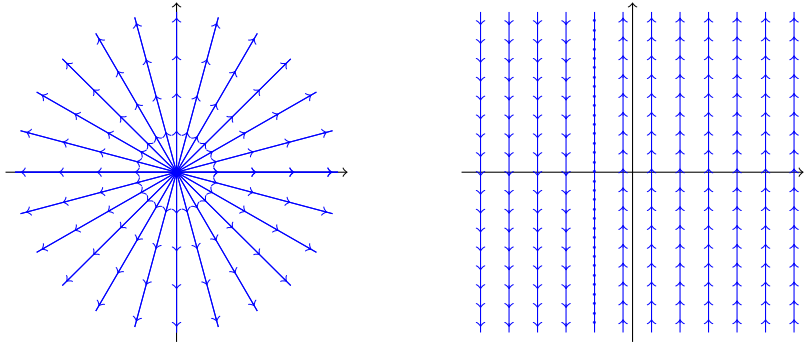
\end{rmk}

\begin{de}
    We said that $\chi$ \emph{fixes only a finite set of lines} if there exists a maximal arrangement fixed by $\chi$. Conversely, we said that $\chi$ \emph{fixes an infinity of lines} if there is no such maximal line arrangement.
\end{de}

We consider the partition $\D(\A)=\D^\infty(\A)\sqcup\D^f(\A)$ following this notion, where $\D^\infty(\A)$ and $\D^f(\A)$ are the sets of elements in $\D(\A)$ fixing only a finite set of lines and fixing an infinite set of lines, respectively. We are interested in the study of this notion in the filtration by degree previously defined, denoting $\D_d^\infty(\A) = \D_d(\A)\cap\D^\infty(\A)$ and $\D_d^f(\A) = \D_d(\A)\cap\D^f(\A)$, for $d\in\NN$.

\begin{rmk}
    The understanding of $\D_d^\infty(\A)$ and $\D_d^f(\A)$ is precisely what is needed to obtain an intrinsic formulation in the article of Llibre at al. \cite{Llibre98}: the polynomial vector field under consideration in any statement belongs to $\D_d^f(\A)$ by hypothesis (see for example Proposition 6 or Theorem 7 in~\cite{Llibre98}).
\end{rmk}

In order to determine the elements of $\D^\infty(\A)$, we introduce a geometrical characterization for vector fields with fix an infinity family of lines.

\begin{de}\label{def:central_parallel}
	A non-null vector field $\chi$ is said to be \emph{central} if there is a point $(x_0,y_0)\in\RR^2$ such that $(x-x_0,y-y_0)$ and $\big(P(x,y),Q(x,y)\big)$ are collinear vectors, for any $(x,y)\in\RR^2$. Otherwise, $\chi$ is said to be \emph{parallel} if there is a $v\in\RR^2$ such that $v$ and $\big(P(x,y),Q(x,y)\big)$ are collinear vectors, for any $(x,y)\in\RR^2$.
\end{de}

Note that there is no vector fields which are simultaneously central and parallel other than the null vector field $\chi=0$.

Let us present a first result relating the combinatorics of an arrangement and the nature of the vector fields in $\D_d(A)$. Let $m(\A)$ be the maximal multiplicity of singular points of $\A$, and let $p(\A)$ be the maximal number of parallel lines.

\begin{thm}\label{thm:bounds_infinite}
	Let $\A$ be a line arrangement and define $\nu_\infty(\A)=\max\{m(\A)-1, p(\A)\}$. If $d<\nu_\infty(\A)$, then $\D_d(\A)$ and $\D_d^\infty(\A)$ are equal.
\end{thm}

This theorem holds directly from the following result:

\begin{propo}\label{Propo:SingParaMax}
	Let $\A$ be an arrangement and let $\chi\in\D_d(\A)$.
	\begin{enumerate}
	    \item If $d<m(\A)-1$ then $\chi$ is a central vector field and $\chi\in \D_d^\infty(\A)$.
	    
	    \item If $d<p(\A)$ then $\chi$ is a parallel vector field and $\chi\in \D_d^\infty(\A)$.
	\end{enumerate}
\end{propo}

\begin{proof}
	We decompose this proof in two cases.

  First, suppose $d+1<m(\A)$. Up to an affine transformation, we may assume that the singular point $P$ of multiplicity $d+2$ of $\A$ is the origin, and that the vertical line $L_x=\{x = 0\}$. Let $L_i=\{y=\alpha_i x\}$ be the $d+2$ lines passing by point $P$. Proposition~\ref{Proposition_CNS} implies that for all $i\in \set{1,\cdots,d+2}$ we have
  \[
    \alpha_i P( y,-\alpha_i y) + Q( y,-\alpha_i y)   = \sum\limits_{n=0}\limits^{d} \Big( \sum\limits_{j=0}\limits^{n} \big(  \alpha_i a_{n-j,j} + b_{n-j,j} \big) (-\alpha_i)^j \Big) y^n  =0,
  \]
  which is equivalent to the system of $(d+2)(d+1)$ equations defined, for all $n\in\set{0,\cdots,d}$ and $i\in\set{1,\cdots,d+1}$, by
  \begin{equation}\tag{$Eq_{(n,j)}$}
      \sum\limits_{j=0}\limits^{n} \big(  \alpha_i a_{n-j,j} + b_{n-j,j} \big) (-\alpha_i)^j = 0
  \end{equation}
  We regroup them in $d+1$ systems $S_n$ formed by the $d+2$ equations (indexed by $i$). These equations are polynomial of degree $n+1$ in $\alpha_i$. We denote by $c_k$ the coefficient of $\alpha^k$, that is $c_0=b_{n,0}$, $c_n=a_{0,n}$ and $c_k=a_{k,n-k}-b_{k-1,n-k+1}$ for $k\in\set{1,n-1}$. If we restrict the system $S_n$ to their $n+2$ first equations, then we remark that the square system in $c_k$ obtained is in fact a Vandermonde system. Since all the $\alpha_i$ are distinct then the system admits a unique solution $c_k=0$. This implies that $a_{0,n}=0$, $b_{n,0}=0$ and $a_{k,d-k}=b_{k-1,d-k+1}$ for $k\in\set{1,d}$. Thus we have $yP(x,y)=xQ(x,y)$, which is a central vector field.

  In a second case, assume that $d<p(\A)$ hence $\A$ has at least $d+1$ parallel lines. Then, without lost of generality, we may assume that these lines are vertical. Let $y_0\in\RR$, from Proposition~\ref{Proposition_CNS} we have that $P(x,y_0)=0$ for $d+1$ different values of $x$. Since $P$ is a polynomial of degree less or equal than $d$, then $P(x,y)=0$ and $\chi$ fixes all the vertical lines.
\end{proof}

Following this study of the appearance of elements in $\D(\A)$ by degree, we can give a first bound for $d_f(\A)$ in terms of combinatorics of the line arrangement.
\begin{cor}
	Let $\A$ be an arrangement in $\RR^2$, then $d_f(\A)\geq \nu_\infty(\A)$.
\end{cor}

\subsection{Characterization of elements in $\D^\infty(\A)$}

In Definition~\ref{def:central_parallel}, we have introduced some classes of vector fields fixing an infinity of lines, defined from a geometric point of view. We prove that any element of $\D^\infty(\A)$ is essentially of this kind of vector fields.

\begin{thm}\label{thm:infty_vf}
	Let $\chi$ be a polynomial vector field fixing an infinity of lines, then $\chi$ is either null, central or parallel. 
\end{thm}

The proof is based on the following lemma, about the number of singular points in a collection of a countable infinity of lines.

\begin{lem}\label{Lemma_Singularities}
	Let $\A_\infty=\set{L_1,L_2,L_3,\ldots}$ be an infinite countable collection of distinct lines in the plane, then we have:
	\begin{equation*}
		\# \Sing(\A_\infty)\in\set{0,1,\infty}.
	\end{equation*}
\end{lem}

\begin{proof}
  We decompose the proof by cases:
	\begin{enumerate}
		\item If all the lines of $\A_\infty$ are parallel, then $\# \Sing(\A_\infty)=0$. 
		\item If all the lines of $\A_\infty$ are concurrent, then $\# \Sing(\A_\infty)=1$.
		\item If $\# \Sing(\A_\infty)\geq 2$, we prove by recurrence that:
			\begin{equation*}
				\forall n \geq 2, \exists k\in \NN^*, \text{ s.t. } \# \Sing(\A_k)\geq n,
			\end{equation*}
			where $\A_i=\set{L_1,L_2,\ldots, L_i}$.
			It is obviously true for $n=2$. Assume that it is true for rank $n$. Since $\A_{n}\subset \A_{n+1}$, then $\# \Sing(\A_{n+1}) \geq \# \Sing(\A_n)$, with equality if $L_{n+1}\cap \A_n\subset \Sing(\A_n)$ (in other terms if $L_{n+1}$ only passes through singular points of $\A_n$). Since there is only a finite number of alignment of points of $\Sing(\A_n)$ then there is an integer $k$ such that $L_{n+k}\cap \A_n \nsubseteq \Sing(\A_n)$. We obtain:
			\begin{equation*}
				\# \Sing(\A_{n+k}) > \# \Sing(\A_n).
			\end{equation*}	
\end{enumerate}
\end{proof}

\begin{proof}[Proof of Theorem~\ref{thm:infty_vf}]
	Let $P(x,y)$ and $Q(x,y)$ be such that $\chi=P\partial_x+Q\partial_y$. We define $\A_\infty=\set{L_1,L_2,L_3,\ldots}$ the set (or a subset) of different lines fixed by $\chi$, and we denoted by $\alpha_i$ the equation of $L_i$. In all what follows, we assume that we are not in the first case (i.e. $(P,Q)\neq (0,0)$).
	
	The vector field $\chi$ fixes only a finite number of lines of $\A_\infty$ point by point. Indeed, $L_i$ is fixed point by point by $\chi$ if and only if $\alpha_i \mid P$ and $\alpha_i \mid Q$. Since $P$ and $Q$ are polynomials then they have finite degree, and only a finite number of $\alpha_i$ can divide them. Assume that these lines are $L_1,\ldots, L_k$.
	
	Denote by $\chi'=P'\partial_x+Q'\partial_y$ the derivation of components $P'=P/(\alpha_1\cdots\alpha_k)$ and $Q'=Q/(\alpha_1\cdots\alpha_k)$. It is clear that $\chi$ and $\chi'$ are collinear vector fields. In this way, if $\chi'$ is central (resp. parallel) then $\chi$ is central (resp. parallel). By construction, the set of points fixed by $\chi'$ (i.e. the common zeros of $P'$ and $Q'$) contains the intersection points of $\A'_\infty=\A\setminus\set{L_1,\ldots,L_k}$. By Lemma~\ref{Lemma_Singularities} we have 3 possible cases:
	\begin{enumerate}
		\item $\# \Sing(\A'_\infty)=0$, then all the lines of $\A'_\infty$ are parallel. By Proposition~\ref{Propo:SingParaMax} $\chi'$ is a parallel vector field.
		\item $\# \Sing(\A'_\infty)=1$, then all the lines of $\A'_\infty$ are concurrent. By Proposition~\ref{Propo:SingParaMax} $\chi'$ is a central vector field.
		\item $\# \Sing(\A'_\infty)=\infty$, then the polynomial $P'$ and $Q'$ have an infinity of zero, which is impossible since $P'$ and $Q'$ are not simultaneously null.
	\end{enumerate}
\end{proof}

\subsection{Influence of the combinatorics in $\D^\infty_d(\A)$: a minimal bound}\label{sec:bound}

The dynamical/geometrical characterization of elements in $\D^\infty(\A)$ obtained in Theorem~\ref{thm:infty_vf} allows us to identify and construct them explicitly. Using this, we determine combinatorially the minimal degree from which $\D^\infty(\A)$ is not empty.

\begin{thm}\label{thm:bounds_finite}
	Let $\A$ be a line arrangement and define $\nu_f(\A)=\min\{|\A|-m(\A)+1, |\A|-p(\A)\}$. If $0<d<\nu_f(\A)$, then $\D_d(A)$ and $\D_d^f(A)$ are equal.
\end{thm}

In order to prove this result, we study each case (presented in Proposition~\ref{Propo:MinRadial} and Proposition~\ref{Propo:MinAffine}), to show that if the degree does not satisfied one of the conditions then we are able to construct explicit elements of $\D^\infty(\A)$. This implies that this lower bound is optimal, then we have:

\begin{cor}
		Let $\A$ be a line arrangement, if $d\geq \nu_f$ then $\D^\infty(\A)\neq\emptyset$.
\end{cor}

\begin{propo}\label{Propo:MinRadial}
	The minimal degree of a non null central vector field fixing a line arrangement $\A$ is:
	\begin{equation*}
		|\A|-m(\A)+1.
	\end{equation*}
\end{propo}

\begin{proof}
	Let $\chi=P\partial_x+Q\partial_y\in\D_d(\A)$ be a central vector field. A line $L=\ker\alpha_L$ is invariant by $\chi$ if we are in one of the following cases:
	\begin{enumerate}
		\item $L$ passes through the center of the vector field.
		\item $\alpha_L$ divides both $P$ and $Q$.
	\end{enumerate}
	The second condition is the most expensive in terms of degree. To minimize this condition, we maximize the first one. 
	Without loss of generality, we may assume that the origin is a singular point of maximal multiplicity. Consider $\A'\subset\A$ the sub-arrangement composed by lines of $\A$ which does not pass by the origin, we have $P(x,y)=\Q_{\A'}p(x,y)$, and $Q(x,y)=\Q_{\A'}q(x,y)$, with $p$ and $q$ such that $yp-xq=0$. The only polynomials of minimal degree verifying this condition are $p(x,y)=x$ and $q(x,y)=y$. Hence, the result holds.
\end{proof}

\begin{propo}\label{Propo:MinAffine}
	The minimal degree of a non null parallel vector fields fixing a line arrangement $\A$ is:
	\begin{equation*}
		|\A|-p(\A).
	\end{equation*}
\end{propo}

\begin{proof}
	Let $\chi=P\partial_x+Q\partial_y\in\D_d(\A)$ be a parallel vector field. A line $L=\ker\alpha_L$ is invariant by $\chi$ if:
	\begin{enumerate}
		\item $L$ is parallel to $\chi$,
		\item $\alpha_L$ divides both $P$ and $Q$.
	\end{enumerate}
	Once again, the second condition is the most expensive in terms of degree and we maximize the first one in degree. Consider $\A'\subset\A$ the sub-arrangement composed by lines of $\A$ which are not parallel to the vector field, then $\Q_{\A'}$ divides both $P$ and $Q$. Thus, the vector field $\chi'=\Q_{\A'}(\partial_x + \partial_y)$ is a vector field of maximal degree fixing $\A$, collinear to the vector $(1,1)$, and the result holds. 
\end{proof}

\begin{proof}[Proof of Theorem~\ref{thm:bounds_finite}]
	Let $\chi \in \D_d(\A)$ with $0<d<\nu_f(\A)$, then $\chi$ does not satisfy the hypotheses of Proposition~\ref{Propo:MinRadial} and Proposition~\ref{Propo:MinAffine}. Thus $\chi$ is neither central nor parallel. Since $d\neq 0$, Theorem~\ref{thm:infty_vf} implies that $\chi\notin \D_d^\infty(\A)$, and then $\chi\in \D_d^f(\A)$.
\end{proof}

Using Theorem~\ref{thm:bounds_infinite} and Theorem~\ref{thm:bounds_finite}, we obtain the following corollary.

\begin{cor}\label{cor:bounds_D(A)}
	Let $\A$ be a line arrangement. Let $\nu(\A)=\min\{\nu_\infty(\A), \nu_f(\A)\}$, if $0<d<\nu(\A)$ then $\D_d(A)=\emptyset$.
\end{cor}
\section{Non combinatoriallity of the minimal finite derivations}\label{sec:combinatorially}

Using the results obtained in Section~\ref{sec:finiteness}, we prove explicitly that $d_f(\A)$ is not determined by the number of lines and singular points counted with multiplicities and, as a more strongest result, by the combinatorial information. For that, we consider two explicit counterexamples of line arrangements. As a first pair, we consider the realizations of configurations $(9_3)_1$ and $(9_3)_2$ described in \cite[p. 102]{Hilbert52}, called the \emph{Pappus} and \emph{non-Pappus} arrangements and denoted by $\P_1$ and $\P_2$ respectively (see~\cite{Suciu01}). Both arrangements have the same weak combinatorics (\emph{i.e.} they share the same number of singularities for each multiplicity). The second pair correspond to Ziegler's arrangement $\mathcal{Z}_1$ (see~\cite{Ziegler89}) and $\mathcal{Z}_2$ a small deformation of $\mathcal{Z}_1$, with same strong combinatorics, \emph{i.e.} $L(\mathcal{Z}_1)\simeq L(\mathcal{Z}_2)$.

\begin{rmk}
	These examples are constructed as the affine parts of the projective arrangements previously described, choosing a line of the arrangement as line at infinity.
\end{rmk}

\subsection{Dependency of weak combinatorics}

The result presented here is a weaker restrictive case of Theorem~\ref{thm:CombiMinDegree}, as a first step to explore the relation between the minimal degree of derivations in $\D(\A)$ and the combinatorics of $\A$.

\begin{thm}\label{thm:SingMinDegree}
	The minimal degree $d_f(\A)$ of a finite polynomial vector field fixing $\A$ is not determined by the number of lines and singular points counted with multiplicities of $\A$.
\end{thm}

\begin{figure}[h]
	\includegraphics[height=6cm]{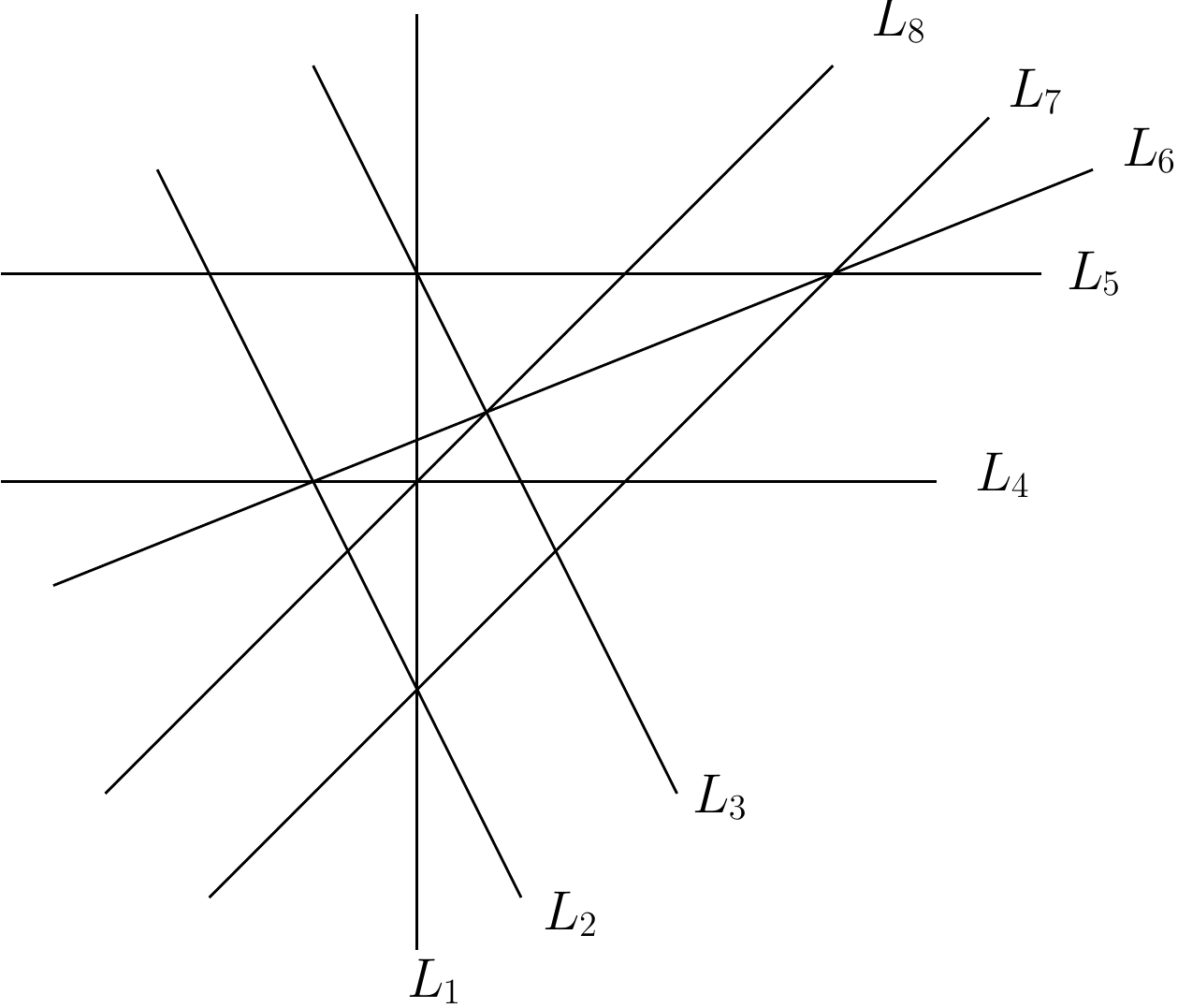}
	\hspace{0.5cm}
	\includegraphics[height=6cm]{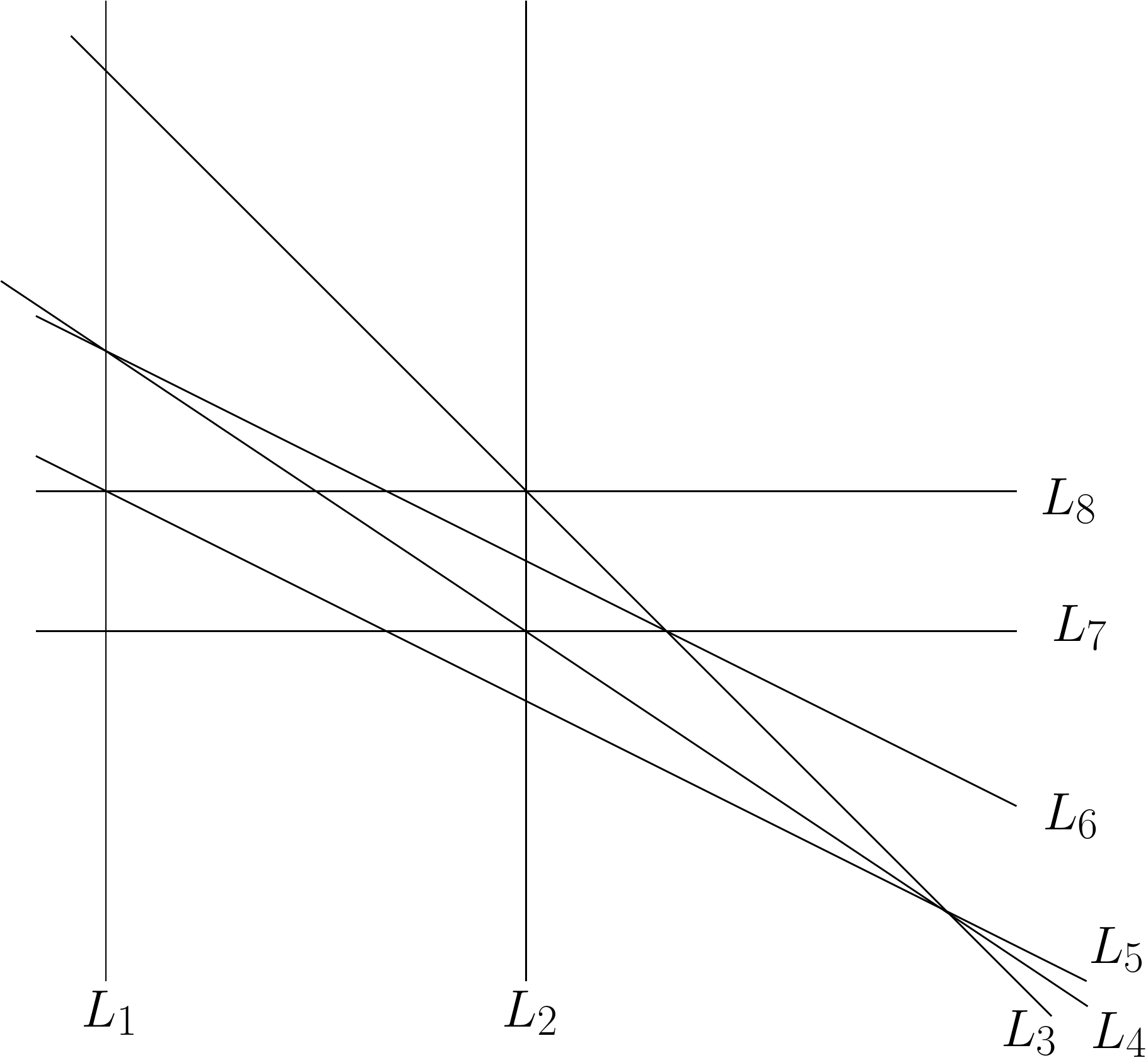}
	\caption{The arrangements $\P_1$ and $\P_2$\label{Fig:Pappus}	}
\end{figure}

In order to prove this theorem, we consider two line arrangements in the plane pictured in Figure~\ref{Fig:Pappus} $\P_1$ (Pappus arrangement) and $\P_2$ (non-Pappus arrangement) defined respectively by:
\begin{equation*}
	\begin{array}{rl}
		\P_1\ :& xy(x-y)(y-1)(x-y-1)(2x+y1)(2x+y-1)(2x-5y+1) \\
		\P_2\ :& xy(x+y)(y+1)(x+3)(x+2y+1)(x+2y+3)(2x+3y+3)
	\end{array}
\end{equation*}
These two arrangements have the same weak combinatorics: 8 lines, 6 triple points and 7 double points. 

\begin{propo}
	The arrangements $\P_1$ and $\P_2$ have not the same combinatorial data, \emph{i.e.} $L(\P_1)\not\simeq L(\P_2)$.
\end{propo}

\begin{proof}
	If we look for lines which posses three triple points and a double point, the only lines in $\P_1$ of this condition are $L_1$ and $L_6$ whose intersection is the common double point, whereas in the line arrangement $\P_2$ we found $L_3$ and $L_4$ with $L_3\cap L_4\cap L_5$ a triple point. 
\end{proof}

Finally, Theorem~\ref{thm:SingMinDegree} holds from the following result.

\begin{propo}\label{propo:ders_Pappus}~
	\begin{enumerate}
		\item For all $i\in\set{1,2,3}$, $\D_i(\P_1)=\emptyset$; and $\D_4(\P_1)\neq\emptyset$,
		\item For all $i\in\set{1,2,3,4}$, $\D_i(\P_2)=\emptyset$; and $\D_5(\P_2)\neq\emptyset$.
	\end{enumerate}
\end{propo}

The proof of this proposition is given in Section~\ref{subsec:proof_propos}.

\begin{proof}[Proof of Theorem~\ref{thm:SingMinDegree}]
    From last proposition, we deduce that $d_f(\P_1)\geq 4$ and $d_f(\P_2)\geq 5$. Using Theorem~\ref{thm:bounds_finite}, we have that $\D_i(\P_1)=\D_i^f(\P_1)$ and $\D_i(\P_2)=\D_i^f(\P_2)$, for any $i\in\set{1,2,3,4,5,6}$. Hence, we obtain that $d_f(\P_1)=4$ and $d_f(\P_2)=5$. This concludes the proof.
\end{proof}

%
%
%

\subsection{Dependency of strong combinatorics}

The main result of this paper is the following:

\begin{thm}\label{thm:CombiMinDegree}
	The minimal degree $d_f(\A)$ of a finite polynomial vectors fields fixing $\A$ is not determined by the combinatorial information of $\A$.
\end{thm}

In order to prove this theorem, consider $\mathcal{Z}_1$ be the affine image of Ziegler arrangement~\cite{Ziegler89}, pictured in Figure~\ref{Fig:Ziegler}. This arrangement verifies a very strong geometric condition: the six triple points of the projective image of $\mathcal{Z}_1$ (considering an additional line in the arrangement: the line at infinity) are contained in a conic $\C$. Hence, we construct a line arrangement $\mathcal{Z}_2$ as a small rational perturbation of Ziegler arrangement, displacing the triple point $L_1\cap L_3\cap L_7$ outside of the conic and preserving the combinatorial data. They are both formed by 8 lines with 4 triples points, 14 doubles points and three pairs of parallel lines. Consider the following equations for $\mathcal{Z}_1$ and $\mathcal{Z}_2$:
\begin{equation*}
	\begin{array}{rl}
		\mathcal{Z}_1\ : &  Z(x,y)(9x-2y+3)(11x+2y+1)(5x+5y-2)\\ 
		\mathcal{Z}_2\ : &  Z(x,y)(21x-4y+7)(19x+4y+1)(10x+10y-5) 
	\end{array}
\end{equation*}
where $Z(x,y)=y(2x+2y+1)(3x+y+1)(8x-y+4)(9x+3y-1)$.

\begin{figure}[ht]
	\includegraphics[height=7cm]{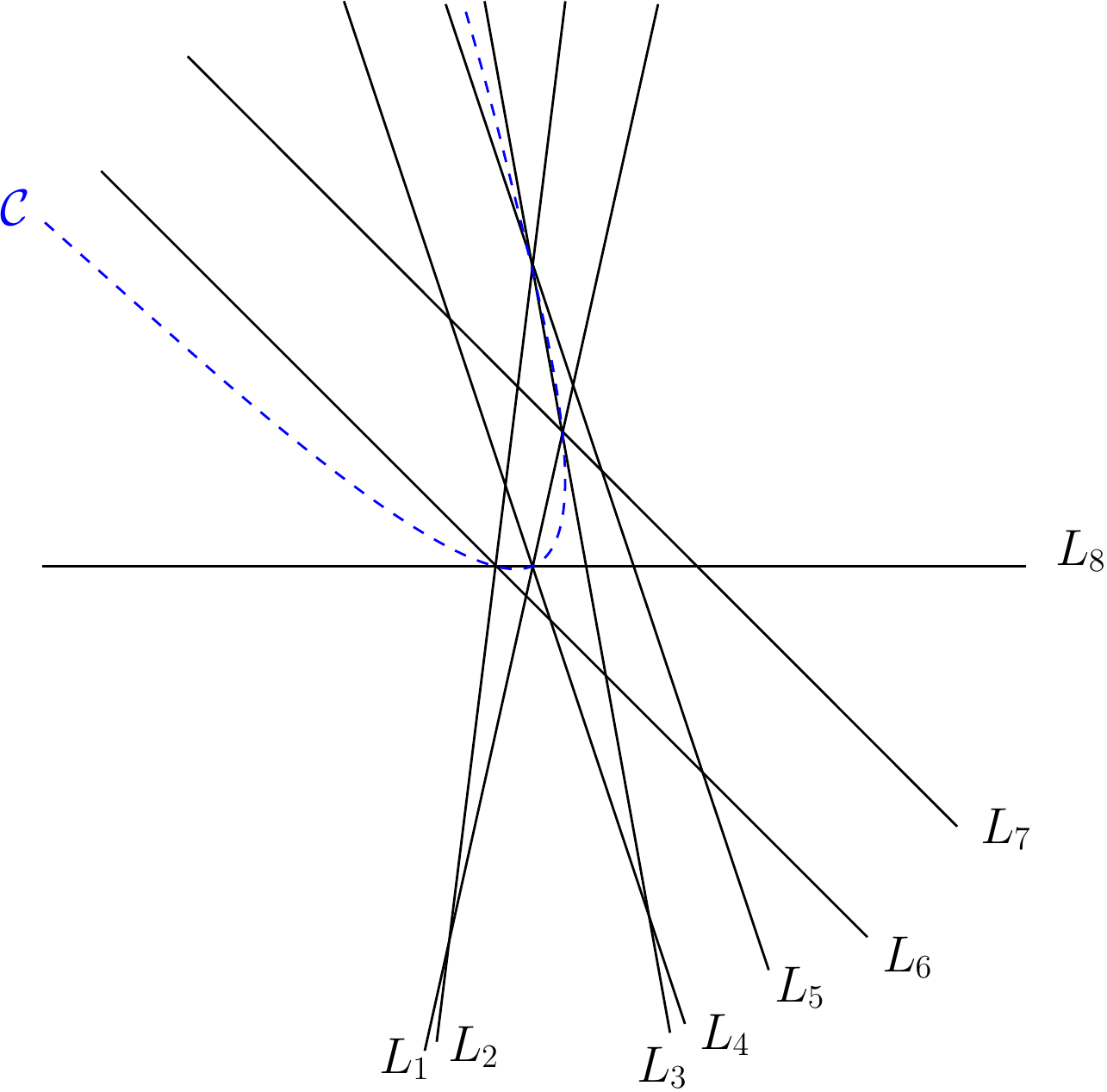}
	\caption{The Ziegler arrangement $\mathcal{Z}_1$ with the conic $\C=\{6x^2 + 2y^2 + 5x + 8xy + 1=0\}$.}\label{Fig:Ziegler}
\end{figure}

\begin{propo}
	The arrangements $\mathcal{Z}_1$ and $\mathcal{Z}_2$ have the same combinatorial information, i.e. $L(\mathcal{Z}_1)\simeq L(\mathcal{Z}_2)$.
\end{propo}

We complete the proof with the following result, discussed in Section~\ref{subsec:proof_propos}. 

\begin{propo}\label{propo:ders_Ziegler}~
	\begin{enumerate}
		\item For all $i\in\set{1,2,3,4}$, $\D_i(\mathcal{Z}_1)=\emptyset$; and $\D_5(\mathcal{Z}_1)\neq\emptyset$,
		\item For all $i\in\set{1,2,3,4,5}$, $\D_i(\mathcal{Z}_2)=\emptyset$; and $\D_6(\mathcal{Z}_2)\neq\emptyset$,
	\end{enumerate}
\end{propo}

\begin{proof}[Proof of Theorem~\ref{thm:CombiMinDegree}]
    As a consequence of last proposition, $d_f(\mathcal{Z}_1)\geq 5$ and $d_f(\mathcal{Z}_2)\geq 6$. But by Theorem~\ref{thm:bounds_finite}, we have that $\D_i(\mathcal{Z}_j)=\D_i^f(\mathcal{Z}_j)$ for any $i\in\set{1,2,3,4,5}$ and $j\in\set{1,2}$. Then $d_f(\mathcal{Z}_1)=5$, which is different of $d_f(\mathcal{Z}_2)$ since it is greater than 6.
\end{proof}

\subsection{Proof of Propositions~\ref{propo:ders_Pappus} and~\ref{propo:ders_Ziegler}}\label{subsec:proof_propos}
    Following Proposition~\ref{Proposition_CNS} and Theorem~\ref{thm:struc_vf}, and using equations (\ref{eq:coeffs_Y}) in Remark~\ref{Rmk_EquationCNS}, the proof of both results is obtained by constructing the matrices $M_{\P_1}, M_{\P_2},M_{\Z_1}$ and $M_{\Z_2}$ for which $\F_4\D(\P_1), \F_5\D(\P_2), \F_5\D(\Z_1)$ and $\F_6\D(\Z_2)$ are the kernels in the respective space of coefficients.
    
    It is easy to see that these matrices have $n(d+1)$ rows and $(d+1)(d+2)$ columns, where $n$ is the number of lines in each line arrangement and $d$ is the degree in the filtration $\F_d$. In order to construct and analyze these matrices, the authors use a set of functions programming over {\tt Sage}~\cite{Sage}, to obtain that $\D_4(\P_1)=\F_4\D(\P_1)\setminus\{0\}$, $\D_5(\P_2)=\F_5\D(\P_2)\setminus\{0\}$, $\D_5(\Z_1)=\F_5\D(\Z_1)\setminus\{0\}$ and $\D_6(\Z_2)=\F_6\D(\Z_2)\setminus\{0\}$. The code source and an appendix with detailed computations can be found in
    \begin{center}
    \url{http://jviusos.perso.univ-pau.fr/pub/combinatorics_vector_fields_appendix.zip}.
    \end{center}

\section{Perspectives}\label{sec:conclu}

The results in this paper can be considered as a first approach concerning the study of the Terao's conjecture about free line arrangements from a dynamical point of view. A line arrangement is called \emph{free} if its corresponding module of derivations is a \emph{free} module. In~\cite{Terao80}, Terao conjectures that freeness of an arrangement is essentially of combinatorial nature: let $\A$, $\A'$ be two line arrangements with same combinatorics (\emph{i.e.} $L(\A)\simeq L(\A')$), if $\A$ is free then $\A'$ is also free and $\D(\A)$, $\D(\A')$ are isomorphic modules. For a precise formulation, we refer to \cite[Chap. 4]{OrlikTerao92}.

The study of $d_f(\A)$, and more generally of $\D_d^f(\A)$, is a first necessary step in this dynamical approach. Furthermore, the Ziegler and non-Ziegler arrangements shows that the set of derivations of a non free arrangement is not determined by the combinatorics and also illustrates the necessity of freeness condition on the arrangement in the Terao's conjecture. The next step will be to dynamically characterize free arrangements.

In~\cite[p.19]{Cartier81}, P. Cartier states that the geometrical interpretation of the freeness condition for a line arrangement is "obscure". His comment relies on the fact that freeness does not seems to be related to any geometrical particularities in the simple case of simplicial line arrangements classified by Gr{\"u}nbaum~\cite{Grunbaum09}. Our previous approach suggest to look for a {\it dynamical} interpretation of freeness. This will be presented in a forthcoming work.

\bigskip
\section*{Acknowledgements}
The authors would like to thank J. Cresson, for the original idea of this paper and all the very helpful discussions and comments. Thanks also to J. Vall{\`e}s for all his explanations about the Terao's conjecture. This work has been developed in the frame of the JSPS-MAE PHC-Sakura 2014 Project. We are very grateful to Professors H. Terao, M. Yoshinaga and T. Abe for their interesting discussions.

\bigskip
\bigskip
\bibliographystyle{alpha}
\bibliography{biblio}

\newcommand{\etalchar}[1]{$^{#1}$}
\def\cprime{$'$}
\begin{thebibliography}{HCV52}

\bibitem[AFV14]{AbeDanieleJean14}
Takuro Abe, Daniele Faenzi, and Jean Vall{\`e}s.
\newblock Logarithmic bundles of deformed {W}eyl arrangements of type $a_2$.
\newblock {\em arXiv:1405.0998v1}, 2014.

\bibitem[AGL98]{Llibre98}
Joan~C. Art{\'e}s, Branko Gr{\"u}nbaum, and Jaume Llibre.
\newblock On the number of invariant straight lines for polynomial differential
  systems.
\newblock {\em Pacific J. Math.}, 184(2):207--230, 1998.

\bibitem[Arn69]{Arnold69}
V.~I. Arnol'd.
\newblock The cohomology ring of the group of dyed braids.
\newblock {\em Mat. Zametki}, 5:227--231, 1969.

\bibitem[Car81]{Cartier81}
Pierre Cartier.
\newblock Les arrangements d'hyperplans: un chapitre de g\'eom\'etrie
  combinatoire.
\newblock In {\em Bourbaki {S}eminar, {V}ol. 1980/81}, volume 901 of {\em
  Lecture Notes in Math.}, pages 1--22. Springer, Berlin-New York, 1981.

\bibitem[DLA06]{Dumortier06}
Freddy Dumortier, Jaume Llibre, and Joan~C. Art{\'e}s.
\newblock {\em Qualitative theory of planar differential systems}.
\newblock Universitext. Springer-Verlag, Berlin, 2006.

\bibitem[FV12a]{DanieleJean12}
Daniele Faenzi and Jean Vall{\`e}s.
\newblock Freeness of line arrangements with many concurrent lines.
\newblock In {\em Eleventh {I}nternational {C}onference {Z}aragoza-{P}au on
  {A}pplied {M}athematics and {S}tatistics}, volume~37 of {\em Monogr. Mat.
  Garc\'\i a Galdeano}, pages 133--137. Prensas Univ. Zaragoza, Zaragoza, 2012.

\bibitem[FV12b]{DanieleJean14}
Daniele Faenzi and Jean Vall{\`e}s.
\newblock Logarithmic bundles and line arrangements, an approach via the
  standard construction.
\newblock {\em arXiv:1209.4934, To appear in J. of the London Math. Soc.},
  2012.

\bibitem[Gr{\"u}09]{Grunbaum09}
Branko Gr{\"u}nbaum.
\newblock A catalogue of simplicial arrangements in the real projective plane.
\newblock {\em Ars Math. Contemp.}, 2(1):1--25, 2009.

\bibitem[HCV52]{Hilbert52}
D.~Hilbert and S.~Cohn-Vossen.
\newblock {\em Geometry and the imagination}.
\newblock Chelsea Publishing Company, New York, N. Y., 1952.
\newblock Translated by P. Nem{\'e}nyi.

\bibitem[LV06]{Llibre06}
Jaume Llibre and Nicolae Vulpe.
\newblock Planar cubic polynomial differential systems with the maximum number
  of invariant straight lines.
\newblock {\em Rocky Mountain J. Math.}, 36(4):1301--1373, 2006.

\bibitem[OS80]{OrlikSolomon80}
Peter Orlik and Louis Solomon.
\newblock Combinatorics and topology of complements of hyperplanes.
\newblock {\em Invent. Math.}, 56(2):167--189, 1980.

\bibitem[OT92]{OrlikTerao92}
Peter Orlik and Hiroaki Terao.
\newblock {\em Arrangements of hyperplanes}, volume 300 of {\em Grundlehren der
  Mathematischen Wissenschaften [Fundamental Principles of Mathematical
  Sciences]}.
\newblock Springer-Verlag, Berlin, 1992.

\bibitem[Ryb11]{Rybnikov11}
G.~L. Rybnikov.
\newblock On the fundamental group of the complement of a complex hyperplane
  arrangement.
\newblock {\em Funktsional. Anal. i Prilozhen.}, 45(2):71--85, 2011.

\bibitem[S{\etalchar{+}}14]{Sage}
W.\thinspace{}A. Stein et~al.
\newblock {\em {S}age {M}athematics {S}oftware ({V}ersion 6.3)}.
\newblock The Sage Development Team, 2014.
\newblock {\tt http://www.sagemath.org}.

\bibitem[Sai80]{Saito80}
Kyoji Saito.
\newblock Theory of logarithmic differential forms and logarithmic vector
  fields.
\newblock {\em J. Fac. Sci. Univ. Tokyo Sect. IA Math.}, 27(2):265--291, 1980.

\bibitem[Suc01]{Suciu01}
Alexander~I. Suciu.
\newblock Fundamental groups of line arrangements: enumerative aspects.
\newblock In {\em Advances in algebraic geometry motivated by physics
  ({L}owell, {MA}, 2000)}, volume 276 of {\em Contemp. Math.}, pages 43--79.
  Amer. Math. Soc., Providence, RI, 2001.

\bibitem[Ter80]{Terao80}
Hiroaki Terao.
\newblock Arrangements of hyperplanes and their freeness. {I}.
\newblock {\em J. Fac. Sci. Univ. Tokyo Sect. IA Math.}, 27(2):293--312, 1980.

\bibitem[Zie89]{Ziegler89}
G{\"u}nter~M. Ziegler.
\newblock Combinatorial construction of logarithmic differential forms.
\newblock {\em Adv. Math.}, 76(1):116--154, 1989.

\bibitem[ZY98]{Xiang98}
Xiang Zhang and Yanqian Ye.
\newblock On the number of invariant lines for polynomial systems.
\newblock {\em Proc. Amer. Math. Soc.}, 126(8):2249--2265, 1998.

\end{thebibliography}

\end{document}